\documentclass[11pt]{amsart}

\usepackage{ amssymb, amsmath, amsthm,  graphicx}
\usepackage{enumerate} % for \bn[1.]
\usepackage[margin=1in]{geometry}
\usepackage{tikz}

\newcommand{\R}{{\mathbb R}}
\newcommand{\C}{{\mathbb C}}

\def\norm#1{\left\|#1\right\|}

\def\inpro#1{\left\langle #1 \right\rangle}

\def\calH{{\mathcal H}}
\def\calW{{\mathcal W}}
\newcommand{\sub}{\subseteq}
\newcommand{\bbF}{\mathbb{F}}
\newcommand{\diagf}{\widetilde{f}}
\newcommand{\fI}{\{f_i\}_{i\in I}}
\newcommand{\Hnwozero}{\calH^n\setminus\{0\}}
\newcommand{\range}{\operatorname{range}}
\newcommand{\Span}{\operatorname{span}}
\newcommand{\tS}{\widetilde{S}}
\newcommand{\sign}{\operatorname{sgn}}

\newtheorem{mythm}{Theorem}[section]
\newtheorem{theorem}{Theorem}[section]
\newtheorem{definition}[mythm]{Definition}
\newtheorem{example}[mythm]{Example}
\newtheorem{corollary}[mythm]{Corollary}
\newtheorem{proposition}[mythm]{Proposition}

\theoremstyle{remark}
\newtheorem{remark}[mythm]{Remark}

\begin{document}

\title{Factor posets of frames and dual frames in finite dimensions$^*$}

\author{Kileen Berry}
\address{Department of Mathematics, University of Tennessee-Knoxville}
\email{berry[at]math[dot]utk[dot]edu}

\author{Martin S. Copenhaver}
\address{Operations Research Center, MIT}
\email{mcopen[at]mit[dot]edu}

\author{Eric Evert}
\address{Department of Mathematics, UCSD}
\email{eevert[at]ucsd[dot]edu}

\author{Yeon Hyang Kim}
\address{Department of Mathematics, Central Michigan University}
\email{kim4y[at]cmich[dot]edu}

\author{Troy Klingler}

\author{Sivaram K. Narayan}
\address{Department of Mathematics, Central Michigan University}
\email{naray1sk[at]cmich[dot]edu}

\author{Son T. Nghiem}
\address{Berea College}
\email{sontnghiem[at]gmail[dot]com}

\thanks{*Research supported  by NSF-REU Grant DMS 11-56890.}

\subjclass[2010]{Primary 42C15, 05B20, 15A03}

\date{\today}

\keywords{Frames, Tight Frames, Factor Poset, $\ell_p$ norm }

\begin{abstract}
We consider frames in a finite-dimensional Hilbert space where frames are exactly the spanning sets of the vector space. A factor poset of a frame is defined to be a collection of subsets of $I$, the index set of our vectors, ordered by inclusion so that nonempty $J \subseteq I$ is in the factor poset if and only if $\{f_i\}_{i \in J}$ is a tight frame. 
We first study when a poset $P\sub 2^I$ is a factor poset of a frame and then relate the two topics by discussing the connections between the factor posets of frames and their duals. Additionally we discuss duals with regard to $\ell^p$ minimization.  \end{abstract}

\maketitle

\section{Introduction}\label{intro}
A frame for a finite dimensional Hilbert space is a redundant spanning set that is not necessarily a basis. The concept of frames was introduced by Duffin and Schaefer \cite{DS1952}. Daubechies \cite{Dau1992} popularized the use of frames. Many of the modern signal processing algorithms used in mobile phone or digital television are developed using the concept of frames. Redundancy in frames plays a pivotal role in the construction of stable signal representations and in mitigating the effect of losses in transmission of signals through communication channels \cite{GKK01, GKV1998}.  
A tight frame is a special case of a frame, which has a reconstruction formula similar to that of an orthonormal basis. Because of this simple formulation of reconstruction, tight frames are employed in a variety of applications such as sampling, signal processing, filtering, smoothing, denoising, compression, image processing, and in other areas. 

A factor poset $\bbF_F$ of a frame $F=\fI$ is the collection of subsets $J\sub I$ such that $\{f_j\}_{j\in J}$ is a tight frame for a finite-dimensional Hilbert space $\calH^n$. We find necessary conditions for a given poset to be a factor poset of a frame. We show that a factor poset is determined entirely by its empty cover (the sets $J\in\bbF_F$ that have no proper subset in $\bbF_F$). Moreover we show that if $P$ is the factor poset of a frame $F\sub\R^2$ then it is also the factor poset of another frame $G\sub\R^2$ whose vectors are multiples of the standard orthonormal basis vectors $e_1$ and $e_2$.

We study the relationship among the factor posets of dual frame pairs. Also we study when the dual frame could be tight and when a dual frame can be scaled to be a tight frame. Finally we consider the group structure among all duals of a frame. It is known that a dual is a canonical dual frame if and only if the $\ell^2$ sum of the frame coefficients is a minimizer among  $\ell^2$ sums of frame coefficients of all dual frames. We find new inequalities among  $\ell^p$ sums of these frame coefficients when $p=1$ and $p>2$.

\section{Preliminaries}

Throughout this paper $\calH^n$ denotes either $\R^n$ to $\C^n$. 
A sequence $F = \{f_i\}_{i=1}^k \sub \calH^n$ is called a frame for $\calH^n$  with frame bounds $A,B > 0$ if for any $f \in \calH^n$, 
\begin{equation}\label{frame}
A \norm{f}^2 \leq \sum_{i=1}^k | \langle f,f_i \rangle |^2 \leq B \norm{f}^2. 
\end{equation}
When $A = B = \lambda$, $F$ is called a $\lambda$-tight frame. When $\lambda =1$, the frame is called a Parseval frame. A unit-norm frame is a frame such that each vector in the frame has norm one. 
For a sequence $F = \{f_i\}_{i=1}^k \sub \calH^n$ define the analysis operator $\theta_F$ from $\calH^n$ to $\calH^k$ by 
\[
 \theta_F x = \sum_{i=1}^k \langle x, f_i \rangle e_i,
\]
where $\{e_i\}_{i=1}^k$ is an orthonormal basis for $\mathcal{H}^k$. 
The adjoint of $\theta_F$, $\theta_F^*:\mathcal{H}^k\to\mathcal{H}^n$, is defined by $\theta_F^*(e_i)=f_i$. The operator $\theta_F^*$ is called the synthesis operator. The frame operator $S_F\,:\, \calH^n \to \calH^n$ associated to $F$ is defined by $S_F = \theta_F^*\theta_F$ and the Gramian operator $E_F$ associated to $F$ is defined by $E_F =  \theta_F\theta_F^*$. The frame operator $S_F$ is a positive definite, self-adjoint, invertible operator and all of its eigenvalues belong to the interval $[A, B]$.

Given a frame $F$, another frame $G=\{g_i\}_{i=1}^k \sub \calH^n$ is said to be a dual frame of 
$F$ if the following reconstruction formula holds: 
$$ f = \sum_{i=1}^k \inpro{f, f_i} g_i \quad \text{for all } f \in \calH^n. $$ 
The canonical dual frame $\widetilde{F}$ associated with $F = \{f_i\}_{i=1}^k$ is defined by $ \widetilde{F} = \{S_F^{-1} f_i \}_{i=1}^k$.

\begin{definition}
For any vector $f=\begin{bmatrix}
f(1)\\\vdots\\f(n)
\end{bmatrix}\in\R^n$, we define the diagram vector of $f$, denoted $\diagf$, by
$$\diagf = \frac{1}{\sqrt{n-1}}\begin{bmatrix}
f^2(1)-f^2(2)\\
\vdots\\
f^2(n-1)-f^2(n)\\
\sqrt{2n}f(1)f(2)\\
\vdots\\
\sqrt{2n}f(n-1)f(n)
\end{bmatrix}\in\R^{n(n-1)}
$$
where the difference of squares $f^2(i)-f^2(j)$ and the product $f(i)f(j)$ occur exactly once for $i<j$, $i=1,\ldots,n-1$.
\end{definition}

\begin{definition}
For any vector $f\in\C^n$, we define the diagram vector $\diagf$ of $f$ to be
\[\tilde{f}=\dfrac{1}{\sqrt{n-1}}\begin{bmatrix}
|f(1)|^2-|f(2)|^2 \\ \vdots \\ |f(n-1)|^2-|f(n)|^2 \\ \sqrt{n}\, f(1)\overline{f(2)} \\ \sqrt{n}\, \overline{f(1)}f(2) \\ \vdots \\ \sqrt{n}\, f(n-1)\overline{f(n)}\\ \sqrt{n}\, \overline{f(n-1)}f(n)
\end{bmatrix}\in \C^{3n(n-1)/2}, \]
where the difference of the form $|f(i)|^2 - |f(j)|^2$ occurs exactly once for $i < j$, $i=1,2,\ldots,n-1$, and the product of the form $f(i)\overline{f(j)}$ occurs exactly once for $i \neq j$.
\end{definition}

Using these definitions a characterization of tight frames in $\calH^n$ is given in \cite{tfs}.

\begin{theorem}[\cite{tfs}]\label{thm:diagCharac}
Let $\fI$ be a sequence of vectors in $\calH^n$, not all of which are zero. Then $\fI$ is a tight frame if and only if $\sum_{i\in I} \diagf_i=0$. Moreover, for any $f,g\in\calH^n$ we have $(n-1)\langle \diagf,\widetilde{g}\rangle = n |\langle f,g\rangle|^2-\|f\|^2\|g\|^2$.
\end{theorem}

\section{Factor Posets}

In \cite{prime}, a tight frame $F=\{f_i\}_{i\in I}$ in $\calH^n$ is said to be \emph{prime} if no proper subset of $F$ is a tight frame for $\calH^n$. One of the main results in \cite{prime} is that for $k\geq n$, every tight frame of $k$ vectors in $\calH^n$ is a finite union of prime tight frames called \emph{prime factors} of $F$. Thus to study the structure of prime factors we use the well-known combinatorial object of posets. A non-empty set $P$ with a partial ordering is called a partially ordered set, or poset. A poset can be represented by a Hasse diagram. We define a poset related to frames as follows:
\begin{definition}
Let $F=\{f_i\}_{i\in I}\subseteq \calH^n\setminus\{0\}$ be a finite frame, where $I=\{1,\ldots,k\}$. The \emph{factor poset} of $F$, denoted $\bbF_F$, is defined to be a collection of subsets of $I$ ordered by set inclusion so that non-empty $J\sub I$ is in $\bbF_F$ if and only if $\{f_j\}_{j\in J}$ is a tight frame for $\calH^n$. We always assume $\emptyset\in\bbF_F$.
\end{definition}

\begin{example}
Let $F=\{e_1,e_2,e_2\}\sub\R^2$ and $I=\{1,2,3\}$. Then $\bbF_F = \{\emptyset,\{1,2\},\{1,3\}\}$ and the Hasse diagram is 
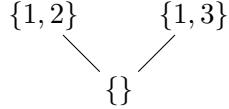
\begin{center}
\begin{tikzpicture}
\node (center) {};
  \node [left of=center](left) {$\{1,2\}$};
 \node [right of=center] (right) {$\{1,3\}$};
 \node [below of=center] (bot) {$\{\}$};
  \draw [ shorten <=-3pt,shorten >=-3pt] (bot) -- (left);
  \draw [ shorten <=-3pt,shorten >=-3pt] (bot) -- (right);
\end{tikzpicture}
\end{center}
\end{example}

\begin{example}\label{eg:secondHasse}
Let $F=\{e_1,e_2,-e_1,-e_2\}\sub\R^2$ and $I=\{1,2,3,4\}$. Then the Hasse diagram of $\bbF_F$ is
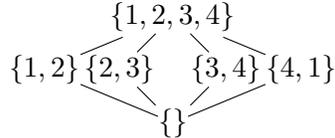
\begin{center}
\begin{tikzpicture}
    \node (top) at (0,0) {$\{1,2,3,4\}$};
    \node [below left  of=top] (mid1) {$\{2,3\}$};
    \node [left of=mid1] (left) {$\{1,2\}$};
 \node [below right of=top] (mid2) {$\{3,4\}$};
 \node [right of=mid2] (right) {$\{4,1\}$};
 \node [below right of =mid1] (bot) {$\{\}$};
 \draw [ shorten <=-4pt,shorten >=-4pt] (top) -- (left);
  \draw [ shorten <=-4pt,shorten >=-4pt] (top) -- (right);
   \draw [ shorten <=-4pt,shorten >=-4pt] (top) -- (mid1);
  \draw [ shorten <=-4pt,shorten >=-4pt] (top) -- (mid2);
  \draw [ shorten <=-4pt,shorten >=-4pt] (bot) -- (left);
  \draw [ shorten <=-4pt,shorten >=-4pt] (bot) -- (right);
    \draw [ shorten <=-4pt,shorten >=-4pt] (bot) -- (mid1);
  \draw [ shorten <=-4pt,shorten >=-4pt] (bot) -- (mid2);
\end{tikzpicture}
\end{center}
\end{example}

The next lemma gives us three equivalent conditions for when the union of elements of $\bbF_F$ is an element of $\bbF_F$.

\begin{proposition}\label{prop:equivCondForInclusion}
Let $F=\{f_i\}_{i\in I}\sub\calH^n\setminus\{0\}$ be a tight frame. Suppose $\bbF_F$ is the factor poset and let $C,D\in\bbF_F$. Then the following are equivalent:
\begin{enumerate}[(i)]
\item $C\cup D\in\bbF_F$,
\item $C\cap D\in\bbF_F$,
\item $C\triangle D\in \bbF_F$, and
\item $C\setminus D\in \bbF_F$.
\end{enumerate}
\end{proposition}
\begin{proof}
By the inclusion-exclusion principle it is easy to verify that for diagram vectors of $F$ that the following hold:
\begin{enumerate}[(a)]
\item $\sum_{\ell \in C\cup D} \widetilde{f_\ell} = \sum_{\ell \in C}\widetilde{f_\ell}+\sum_{\ell \in D}\widetilde{f_\ell}-\sum_{\ell \in C\cap D}\widetilde{f_\ell}$, 
\item $\sum_{\ell \in C\cup D}\widetilde{f_\ell}=\sum_{\ell \in C\setminus D}\widetilde{f_\ell}+\sum_{\ell \in D \setminus C}\widetilde{f_\ell}+\sum_{\ell \in C\cap D}\widetilde{f_\ell}$, and
\item $\sum_{\ell \in C}\widetilde{f_\ell}= \sum_{\ell \in C\setminus D}\widetilde{f_\ell}+\sum_{\ell \in C\cap D}\widetilde{f_\ell}$.
\end{enumerate}
Since $C,D\in\bbF_F$, using Theorem \ref{thm:diagCharac} we have $\sum_{\ell \in C}\widetilde{f}_\ell=\sum_{\ell\in D}\widetilde{f}_\ell=0$. Hence, from (a) we see that $(i)\iff (ii)$. By the definition of symmetric difference $C\triangle D$, the implication $(i)\iff(ii)$ and (b) above it follows that $(i)\implies(iii)$. Conversely, if $(iii)$ holds, then from (b) we have $\sum_{\ell\in C\cup D}\widetilde{f}_\ell = \sum_{\ell\in C\cap D}\widetilde{f}_\ell$. But from (a), when $C,D\in\bbF_F$ we have $\sum_{\ell \in C\cup D} \widetilde{f}_\ell = -\sum_{\ell\in C\cap D}\diagf_\ell$. Hence $(iii)\implies (ii)$. Hence $(i)-(iii)$ are equivalent. Using (c) above we conclude that $(iv)\iff(ii)$.
\end{proof}

The above proposition gives some necessary conditions for a given poset to be  a factor poset of a frame.

\begin{proposition}
Let $F=\fI\sub\Hnwozero$ be a finite frame with corresponding factor poset $\bbF_F$. Then for any $m\in \mathbb{N}$ with $m\geq |I|=k$ there exists a frame sequence $G = \{g_j\}_{j\in J}\sub\Hnwozero$ where $J = \{1,\ldots,m\}$ such that $\bbF_F = \bbF_G$.
\end{proposition}
\begin{proof}
We show that there exists some $g_{k+1}\in\Hnwozero$ so that $G = F\cup \{g_{k+1}\}$ satisfies $\bbF_F=\bbF_G$. Consider $S = \left\{-\sum_{\ell\in L} \diagf_\ell : \emptyset\subsetneq L\sub I\right\}$. This is a finite collection of vectors in $\R^{n(n-1)}$ or $\C^{3n(n-1)/2}$. Now select $g_{k+1}\in \Hnwozero$ so that $\widetilde{g}_{k+1}\notin S$. This completes the proof.
\end{proof}

\begin{definition}
For a frame $F=\fI\sub\calH^n$ and its factor poset $\bbF_F$, we define the empty cover of $\bbF_F$, denoted $EC(\bbF_F)$, to be the set of $J\in \bbF_F$ which cover $\emptyset\in\bbF_F$, that is,
$$EC(\bbF_F) = \{J\in \bbF_F: J\neq\emptyset\text{ and } \not\exists J'\in \bbF_F \text{ with } \emptyset \subsetneq J'\subsetneq J\}.$$
\end{definition}

\begin{example}
Let $F = \{e_1,e_2,-e_1,-e_2\}\sub\R^2$. As seen from Example \ref{eg:secondHasse},
$$EC(\bbF_F) = \{\{1,2\},\{2,3\},\{3,4\},\{4,1\}\}.$$
\end{example}

We now show that a factor poset is entirely determined by its empty cover.

\begin{proposition}\label{prop:decomp}
Let $F=\fI\sub\calH^n$ be a finite frame. If $\bbF_F$ is the factor poset of $F$, then for any non-empty $J$ in $\bbF_F\setminus EC(\bbF_F)$, there exists $J_1,J_2\in \bbF_F$ with $J_1\subsetneq J$ and $J_2\subsetneq J$ so that $J_1\cap J_2=\emptyset$ and $J_1\sqcup J_2 = J$.
\end{proposition}
\begin{proof}
Let $J\in \bbF_F\setminus EC(\bbF_F)$ be a non-empty set. There must exist some non-empty $J_1\in \bbF_F$ so that $J_1\subsetneq J$, otherwise $J\in EC(\bbF_F)$. Using Proposition \ref{prop:equivCondForInclusion} it is easy to see that $J_2:=J\setminus J_1\in \bbF_F$. By assumption on $J$ and $J_1$ we see that $J_2$ is non-empty. Hence $J= J_1\sqcup J_2$. 
\end{proof}

\begin{corollary}\label{cor:ECdeterminesEverything}
Let $F=\fI,G=\{g_i\}_{i\in I}$ be finite frames in $\calH^n$ with factor posets $\bbF_F$ and $\bbF_G$, respectively. Then $EC(\bbF_F) = EC(\bbF_G)$ if and only if $\bbF_F=\bbF_G$. 
\end{corollary}

\begin{proof}
It is obvious that if $\bbF_F=\bbF_G$ then the empty covers are equal, so we restrict our attention to the other direction. It suffices to show that the factor poset of the frame is entirely determined by its empty cover. Let $S = EC(\bbF_F)\cup \{\emptyset\}$ for a frame $F=\fI\sub \calH^n$ with $\bbF_F$ as its factor poset. For every $J_1,J_2\in S$, if $J_1\cap J_2\in S$ then append $J_1\cup J_2$ to $S$. Repeat this process until no more unions can be added. This process must terminate after finitely many iterations since $I$ is finite. Clearly the new collection of sets which we again denote by $S$ is contained in $\bbF_F$. From Proposition \ref{prop:decomp} the reverse inclusion holds. Therefore the factor poset $\bbF_F$ is determined by $EC(\bbF_F)\cup\{\emptyset\}$. The desired result follows.
\end{proof}

As a consequence of Corollary \ref{cor:ECdeterminesEverything} we get an alternate proof of the following result from \cite{prime}. 
\begin{corollary}\label{cor:primeDecomp}
Every tight frame $F = \fI\sub\Hnwozero$ can be written as a union of prime tight frames.
\end{corollary}
The proof of Corollary \ref{cor:primeDecomp} follows from observing that if $J\in EC(\bbF_F)$ then $\{f_j\}_{j\in J}$ is a prime tight frame. An important case of factor posets occurs when we consider a tight frame $F = \fI\sub \Hnwozero$. Note that when $F$ is tight $\emptyset,I\in \bbF_F$.
\begin{definition}
Suppose $F=\fI\sub \Hnwozero$ is a frame. Let $\chi(F)$ denote the sequence indexed by $I$ where $\chi(F)(i)$ is the number of times $i$ occurs in $EC(\bbF_F)$ for each $i\in I$. We call $\chi(F)$ the characteristic of $F$. If $\chi(F)(i)=m$ for all $i\in I$ then $F$ is said to have uniform characteristic.
\end{definition}

\begin{example}
Let $F = \{e_1,e_1,e_2\}\sub \R^2$. Then $EC(\bbF_F) = \{\{1,3\},\{2,3\}\}$ and $\chi(F) = (1,1,2)$. Hence $\chi(F)$ need not be uniform.
\end{example}
\begin{proposition}\label{prop:posUniCharac}
If $F = \fI\sub\Hnwozero$ has positive uniform characteristic, then $F$ is a tight frame.
\end{proposition}

\begin{proof}
Suppose that $F$ has uniform characteristic $m>0$. Let $S_1,\ldots,S_h$ be the elements of $EC(\bbF_F)$. Then $\sum_{i\in S_q}\diagf_i = 0$ for $q=1,\ldots,h$. So $\sum_{q=1}^h\sum_{i\in S_q}\diagf_i=0$. Since $j\in I$ occurs in $EC(\bbF_F)$ $m$ times it follows that $\diagf_j$ occurs $m$ times in the sum $\sum_q\sum_{i\in S_q}\diagf_i=0$. Hence
$$\sum_{j\in I}\diagf_j = \frac{1}{m}\left(\sum_q\sum_{i\in S_q}\diagf_i\right)=0.$$
Hence $F$ is a tight frame.
\end{proof}
\begin{remark}
The condition in Proposition \ref{prop:posUniCharac} is sufficient but not necessary. Consider the frame $F=\{e_1,e_1,e_2,e_2,e_1+e_2,e_1-e_2\}$ which is a tight frame but $\chi(F) = (2,2,2,2,1,1)$.
\end{remark}

The following theorem states that give a factor poset $P$ of a frame in $\R^2$ we can find another frame with vectors parallel to $e_1$ and $e_2$ (the standard othonormal basis) whose factor poset is also $P$.

\begin{theorem}\label{thm:projToONB}
Let $F=\fI\sub\R^2\setminus\{0\}$ be a finite frame with $I=\{1,\ldots,k\}$. Then there exists a frame $G=\{g_i\}_{i\in I}$ whose vectors are scaled multiples of the standard orthonormal basis $e_1$ and $e_2$ such that $\bbF_F=\bbF_G$.
\end{theorem}
\begin{proof}
Let $\{J_\ell:1\leq \ell\leq 2^k\}$ be an enumeration of $2^I$ and let
$$2^I\setminus \bbF_F = \left\{J_{\ell_r} : \sum_{i\in J_{\ell_r}}\diagf_i\neq 0\right\}.$$
Consider a projection $P$ of rank 1 on $\R^2$ such that $\range(P)\neq \left(\Span\left\{\sum_{i\in J_{\ell}}\diagf_i\right\}\right)^\perp$ for any $J_{\ell}$. Let $\tS=\{P(\diagf_i) :1\leq i\leq k\}$ and $S$ be the a set of vectors of cardinality $k$ in $\R^2$ whose diagram vectors are equal to the set $\tS$.

We now claim that $\sum_{i\in J_\ell} \diagf_i=0$ if and only if $\sum_{i\in J_\ell}P(\diagf_i)=0$. The forward implication is clear. Now assume $\sum_{i\in J_\ell} P(\diagf_i)=0$. Then $\sum_{i\in J_\ell}\diagf_i\in \ker(P)$. By the choice of $P$, $\ker(P)\cap \left(\Span\left\{\sum_{i\in J_\ell} \diagf_i\right\}\right) = \{0\}$ for all $J_\ell\in 2^I$. Therefore, $\sum_{i\in J_\ell} \diagf_i\in \ker(P)$ iff $\sum_{i\in J_\ell}\diagf_i=0$. This proves the claim.

Now assume that $\bbF_F$ contains something other than the empty set. Then $F$ has a tight subframe. Hence there exists some $J'\in 2^I$ such that $\sum_{i\in J'}\diagf_i=0$. This implies $\sum_{i\in J'}P(\diagf_i)=0$. Because $\diagf_i\neq0$ for each $i\in J'$ we know $P(\diagf_i)\neq0$. By assumption $\range(P)=\Span\{v\}$ where $v$ is a unit vector. Then there exist nonzero scalars $\{\alpha_i\}_{i\in J'}$ such that $\alpha_iv = \diagf_i$ for each $i\in J'$. Since $0=\sum_{j\in J'} P(\diagf_j) = \sum_{j\in J} \alpha_jv$ and $\alpha_j\neq 0$ we have $s,t\in J'$ such that $\sign(\alpha_s) = - \sign(\alpha_t)$. Since $\diagf_s=\alpha_sv$ and $\diagf_t=\alpha_tv$ we must have corresponding vectors in $S$ that are nonzero and orthogonal. Since any two nonzero orthogonal vectors span $\R^2$, the vectors in $S$ must span $\R^2$ and hence form a frame.

Suppose $J_\ell\in \bbF_F$. Then $\sum_{i\in J_\ell} \diagf_i=0$. From the claim $\sum_{i\in J_\ell} \diagf_i=0$ iff $\sum_{i\in J_\ell} P(\diagf_i)=0$. Hence $J_\ell\in \bbF_S$. The reverse direction is similar, and thus $\bbF_F=\bbF_S$.

Since $\operatorname{rank}(P)=1$, there exists a unitary operator $U$ such that $Uv=e_1$. Hence
$$UP(\diagf_i) = \begin{bmatrix}
\lambda_i\\0
\end{bmatrix}$$
for some $\lambda_i\in\R$. Define $g$ as follows:
$$g_i := \left\{\begin{array}{ll}
[\sqrt{\lambda_i}\;\;\;\;\;\;0]^T,&\lambda_i\geq0\\
\begin{bmatrix}
0&\sqrt{-\lambda_i}
\end{bmatrix}^T&\lambda_i<0.
\end{array}\right.$$
Let $G=\{g_i\}_{i\in I}$. It easily follows that $UP(\diagf_i)=\widetilde{g}_i$. Moreover $\sum_{i\in J_\ell}P(\diagf_i)=0$ if and only if $\sum_{i\in J_\ell} UP(\diagf_i)=0$. Therefore $\bbF_G=\bbF_S=\bbF_F$.
\end{proof}

\begin{remark}
Based on the above Theorem \ref{thm:projToONB} we propose the following Inverse Factor Poset Problem: given a poset $P\sub 2^I$, does there exist a frame $F\sub \R^n$ such that $\bbF_F=P$?
\end{remark}

\section{Dual frames}

For a given frame $F$, we define the the set $\calW_F$ as follows:
\[
\calW_F:=
\left\{
\begin{array}{rclcl}
W=\left[
\begin{array}{rclcl}
\leftarrow &w_1& \rightarrow \\
& \vdots & \\
\leftarrow &w_n& \rightarrow \\
\end{array}
\right]
: \, \overline{w}_i \in \ker (\theta_F^*)
\end{array}
\right\}
\]

Then, by the result in \cite{ Li1995, CPX12},  we have the following proposition.

\begin{proposition}\label{dual}
Let  $F$ be a frame for $\calH^n$. Then  any dual frame to a frame $F$ can be expressed as columns of the matrix  
\begin{equation}\label{dualeq}
S_F^{-1}\theta^*_F + W, 
\end{equation}
for some $W \in \calW_F$.\end {proposition}

For a given frame  $F$, let $\mathcal{G}=\{S_F^{-1} \theta_F^*+W: W\in \mathcal{W_F} \}$ be the set of all matrices whose columns form duals of $F$.  Define the operation $\oplus: \mathcal{G}\times\mathcal{G}\rightarrow \mathcal{G}$ by 
$(S_F^{-1} \theta_F^*+W_1) \oplus (S_F^{-1} \theta_F^*+W_2):=S_F^{-1} \theta_F^*+W_1+W_2$.
 
\begin{theorem} 
Let  $F$ be a frame and let $\mathcal{G}=\{S_F^{-1} \theta_F^*+W: W\in \mathcal{W_F} \}$. Then
$(\mathcal{G},\oplus)$ defines an abelian group.
\end{theorem}

\begin{proof}
For any $W_1, W_2 \in \mathcal{W}$,  since $\ker(\theta_F^*)$ is a vector space, 
$W_1 + W_2 \in \mathcal{W}$, 
which implies that  $\mathcal{G}$ is closed under $\oplus$.  
Associativity and commutativity follow from associativity and commutativity of matrix addition, and the identity is given by $S_F^{-1} \theta_F^*$. 
Each element $S_F^{-1} \theta_F^*+W \in \mathcal{G}$ has an inverse $S_F^{-1} \theta_F^*-W$. 
This completes the proof. 
\end{proof}

\begin{proposition}
Let $F$ be a tight frame. Suppose that  $G \in \{S_F^{-1} \theta_F^*+W: W\in \mathcal{W_F} \}$ is a matrix whose columns form a tight frame.  
Then the subgroup generated by $G$ is contained in the set of matrices whose columns form tight duals of $F$.
\end{proposition}

\begin{proof}
Let $S_f = \lambda I_n$. Then $G = \frac{1}{\lambda} \theta_F^* +W$ for some $W \in \calW_F$, where 
$WW^* = \alpha I_n$ for some $\alpha \in \R$. 
If $H =  \frac{1}{\lambda} \theta_F^* +mW $ for some $m \in \mathbb{N}$, then 
$H^*H =  (\frac{1}{\lambda}  +m^2\alpha) I_n$, which  completes the proof. 
\end{proof}

The following example shows that in general, it is not true that the set of matrices in $\mathcal{G}$ whose columns form a tight frame is a subgroup of $(\mathcal{G}, \oplus)$. 

\begin{example}
 Let $F$ be the frame where the synthesis operator  $\theta_F^*$ is given by
$$
\theta_F^*=
\left[
\begin{array}{ccccc}
1 & 0 & 0 & 0 \\
0 & 1 & 0 & 0
\end{array}
\right].
$$
Then  the set of all matrices whose columns form a dual of $F$ is  
$$
\mathcal{G}=\left\{
\begin{array}{ccccc}
\left[
\begin{array}{ccccc}
1 & 0 & a & b \\
0 & 1 & c & d
\end{array}
\right]
:a,b,c,d \in \mathbb{R}
\end{array}
\right\}.
$$
We consider the following two matrices:
$$
A=
\left[
\begin{array}{ccccc}
1 & 0 & 0 & 1 \\
0 & 1 & 1 & 0
\end{array}
\right],
\quad 
B=
\left[
\begin{array}{ccccc}
1 & 0 & 1 & 0 \\
0 & 1 & 0 & 1
\end{array}
\right].
$$
Then $A, B \in \mathcal{G}$, and the columns of $A$ and $B$ form a tight dual of $F$. However, the columns of
$$
A \oplus B=
\left[
\begin{array}{ccccc}
1 & 0 & 1 & 1 \\
0 & 1 & 1 & 1
\end{array}
\right]
$$
do not form a tight dual of $F$. 
\end{example}

We study the relationship between the factor posets for a tight frame and its canonical dual. 
 Recall that an isomorphism on posets $(P_1, \leq_1), (P_2, \leq_2)$ is a bijection $\phi: P_1 \rightarrow P_2$ so that $\phi(a) \leq_2 \phi(b)$ if and only if $a \leq_1 b$ for all $a,b \in P_1$.
We define a stronger notion of order isomorphism. We let $S_m$ denote the symmetric group on $m$ elements.
 
 \begin{definition}
We say that two factor posets $\bbF_F$ and $\bbF_G$ corresponding to frames $F= \{f_i\}_{i\in I}$ and $G= \{g_j\}_{j\in J}$ are \textit{strongly isomorphic} if there exists some $m \in \mathbb{N}$ and some $\eta \in S_m$, such that $\eta(\bbF_F)=\sigma(\bbF_G)$, where $\eta(\bbF_F)=\{\eta(J'):J' \in \bbF_F\}$ and $\eta(J')=\{\eta(j):j \in J'\}$.
\end{definition}

If $F_J$ is a tight subframe of a $\lambda$-tight frame $F = \{f_i\}_{i \in I}$ for some $J \sub I$, then 
$\sum_{i \in J} \tilde{f_i} = 0$ so that we have $\sum_{i \in J} S^{-1}_{F}\tilde{f_i} =  \sum_{i \in J} \frac{1}{\lambda^2}\tilde{f_i} = 0$. This implies that $\{S_F^{-1} f_i \}_{j \in J}$ is a tight frame. Thus we have the following theorem. 

\begin{theorem}
A tight frame $F$ and its associated canonical dual frame $\{S_F^{-1} f_i \}_{i=1}^k$ have the same factor posets. 
\end{theorem}

This result does not hold true for non-tight frames and their canonical duals.
For example, 
the factor poset of the following frame $F$ and its canonical dual are not strongly isomorphic
\[
F =\left\{\left[
\begin{array}{ccccc}
1 \\
0 \\
\end{array}
\right],
\left[
\begin{array}{ccccc}
0 \\
1 \\
\end{array}
\right],
\left[
\begin{array}{ccccc}
\frac{ 3989\sqrt{15912321}}{100}\\
0 \\
\end{array}\right],
\left[
\begin{array}{ccccc}
0\\
\frac{3989}{10}\\
\end{array}
\right]\right\}.
\]

\begin{proposition}
There exist a frame $F$ such that no dual $G$ of $F$ has a  factor poset structure that is strongly isomorphic to $\mathbb{F}_F$.
\end{proposition}

\begin{proof} 
Let $F=\{e_1, e_1, e_2\}$ be a frame for $\R^2$, where $\{e_1,e_2\}$ are the standard orthonormal basis of $\R^2$. 
Let $G$ be an arbitrary dual of $F$, then by Proposition \ref{dual}, 
\[\theta^*_G = 
\left[
\begin{array}{ccccc}
\frac{1}{2} + a & \frac{1}{2} - a & 0 \\
b & -b & 1
\end{array}
\right],
\]
for some
$a,b \in \mathbb{R}$.
We consider the poset $ \bbF_F =\{\emptyset,  \{1,3\} \{2,3\} \} $. 
We know that two vectors in $\mathbb{R}^2$ form a tight  frame if and only the vectors are orthogonal and are of equal norm. If $F$ and $G$ have strongly isomorphic poset structures, then two of the vectors of $G$ must be orthogonal to the remaining vector in $G$, and all three vectors have equal norms. This is impossible. 
\end{proof} 

From the dual expression given in (\ref{dualeq}), we obtain the a characterization of tight duals of a tight frame. The following result is remarked in \cite{KKL12}; the proof given here is different. 

\begin{theorem}
Let $F$ be a $\lambda$-tight frame with $k$ frame elements for $\mathcal{H}^n$.  If $k < 2n$, then the canonical dual is the only tight dual of $F$. If $k \geq 2n$ then $F$ has an alternate dual that is tight.
\end{theorem}

\begin{proof}
Let $G$ be a dual frame of $F$. 
Since 
$\theta_G^* = \frac{1}{\lambda} \theta_F^* + W$ for some $W \in \calW$, we have that 
$\theta_G^* \theta_G = \frac{1}{\lambda} I_n + WW^* $. 
This implies that $G$ is tight if and only if $WW^*= \alpha I_n$ for some $\alpha \in \R$.  
If $k < 2n$, since $\dim(\ker(\theta_F^*)) < n$, we have $\alpha =0$. 
This implies that if $k < 2n$, then the canonical dual is the only tight dual of $F$. 
Let $k \ge 2n$ and 
let  $\{\overline{w}_j \}_{j=1}^{k-n}$ be an orthonormal basis for $\ker(\theta ^*)$. 
Then for any $ \alpha \in \calH \backslash  \{0\}$,  we consider 
$$ W=
\alpha \left[
 \begin{array}{rclcl}
\leftarrow &w_1& \rightarrow \\
& \vdots & \\
\leftarrow &w_n& \rightarrow \\
\end{array}
\right]
.
$$
Since $W \in \calW$, we have 
 $( \frac{1}{\lambda} \theta_F^* + W) ( \frac{1}{\lambda} \theta_F^* + W)^* = 
(\frac{1}{\lambda} + |\alpha|^2) I_n$, which implies that the columns of $ \frac{1}{\lambda} \theta_F^* + W$ forms a tight dual frame of $F$. 
\end{proof}

We remark that for any tight frame,  the canonical dual frame has the smallest frame bound among all  tight dual frames.  
The next result provides us a simple construction of a dual frame from a dual of its subframe. 

\begin{theorem}
Let $F =\{ f_i\}_{i\in I}$ be a frame for $\calH^n$ and $H = \{f_i\}_{j \in J}$ be a subframe of $F$. 
If $K = \{ g_i\}_{i \in J}$ is a dual of $H$, then $G = \{g_i\}_{i \in I}$, where 
$$
g_i =
\begin{cases}
g_i, &  \text{ if  }  i \in J  \cr
0, &  \text{ if  } i \in I \backslash J
\end{cases}, 
$$
is  a dual of $F$
\end{theorem}
\begin{proof}
The proof follows by 
$$ \theta_G^* \theta_F = \theta_K^* \theta_H = I_n. $$
\end{proof}

Since any frame has a basis subframe, we have the following.

\begin{corollary}
If $F = \{ f_i \}_{i}^k$ is a frame for $\calH^n$, then there exists a dual of $F$ 
consisting of $n$ basis vectors for $\calH^n$ and $(k-n)$ zero vectors. 
\end{corollary}

\begin{corollary}
If a frame in $\calH^n$ has a tight subframe, then it has a tight dual.
\end{corollary}

It is known that  a frame $F$ is a basis  if and only if a frame  $F$ has a unique dual frame \cite{HKLW07}.  We observe that a frame which is a basis has also a connection with scalability. 
To this end, we define a frame $F=\{f_i \}_{i=1}^k$ for  $\calH^n$ to be scalable if 
there exists scalars  $\{c_i \}_{i=1}^k$ such that $\{c_i f_i \}_{i=1}^k$ is a Parseval frame. 

\begin{proposition}
Let $F$ be a basis for $\calH^n$. Then $F$ is scalable if and only if $F$ is an orthogonal basis. 
\end{proposition}
\begin{proof}
Let $F=\{ f_i \}_{i=1}^n$ be a basis. 
If there exists scalars  $\{c_i \}_{i=1}^n$ such that $H = \{c_i f_i \}_{i=1}^n$ is a Parseval frame, then 
$\theta_H^*$ is an unitary matrix, which implies that $F$ is an orthogonal basis. The converse is clear. 
\end{proof}

 %%%%% ell_p section
\section{$\ell^p$ norm of the frame coefficients}

It is well known that 
the $\ell^2$ norm of the frame  coefficients with respected to the canonical dual is smaller that the $\ell^2$ norm of the frame coefficients with respect to any other dual. Moreover, this $\ell^2$-minimization characterizes the canonical dual of a frame.

\begin{proposition}[\cite{HKLW07}]\label{min2}
Let $\{f_i\}_{i=1}^k$ be a frame for $\calH^n$ and let $\{g_i\}^k_{i=1}$ be a dual frame of $\{f_i\}_{i=1}^k$. Then $\{g_i\}_{i=1}^k$ is the canonical dual if and only if
\[
\sum_{i=1}^k |\langle f,g_i\rangle|^2 \leq \sum_{i=1}^k |\langle f,h_i\rangle|^2 \quad  \forall f\in \mathcal{H}^n,
\]
for all frames $\{h_i\}_{i=1}^k$ which are duals of $\{f_i\}_{i=1}^k$.
\end{proposition}

Using  Newton's
generalized binomial theorem  and the H$\ddot{\rm o}$lder's inequality, 
for any   two sequences  $x = \{x_i\}_{i=1}^k$  and $y = \{y_i\}_{i=1}^k$ and $p \in (1, \infty)$, we have 
\begin{equation}\label{holder}
 \|x\|_p \le \|x\|_1 \le k^{(1-\frac{1}{p})} \|x\|_p, 
 \end{equation}
where $\|x\|_p =(\sum_{i=1}^k |x_i|^p)^{1/p}$.
The right-side inequality in (\ref{holder}) with $p=2$ and Proposition  \ref{min2} gives us the following result:

\begin{proposition}\label{min1}
Let $\{f_i\}_{i=1}^k$ be a frame for $\calH^n$ and let $\{h_i\}^k_{i=1}$ be a dual frame of $\{f_i\}_{i=1}^k$. If  $\{g_i\}_{i=1}^k$ is the canonical dual, then 
\[
\sum_{i=1}^k |\langle f,g_i\rangle| \leq \sqrt{k}  \sum_{i=1}^k |\langle f,h_i\rangle| \quad  \forall f\in \mathcal{H}^n. 
\]
\end{proposition}

From the inequalities  and   Proposition  \ref{min1}, for any $p \in (1, \infty)$, we obtain 
\[
 \sum_{i=1}^k |\langle f,g_i\rangle|^p \leq  k^{(\frac{3}{2}p - 1)} 
 \sum_{i=1}^k |\langle f,h_i\rangle|^p \quad  \forall f\in \mathcal{H}^n. 
\]

If $p>2$, we have the better estimation. 
\begin{theorem}\label{minp}
Let $\{f_i\}_{i=1}^k$ be a frame for $\calH^n$ and let $\{h_i\}^k_{i=1}$ be a dual frame of $\{f_i\}_{i=1}^k$. If  $\{g_i\}_{i=1}^k$ is the canonical dual, then  for any $p\in (2, \infty)$, we have 
\[
 \sum_{i=1}^k |\langle f,g_i\rangle|^p \leq  k^{(\frac{1}{2}p - 1)} 
 \sum_{i=1}^k |\langle f,h_i\rangle|^p \quad  \forall f\in \mathcal{H}^n. 
\]

\end{theorem}
\begin{proof}
First observe that the right-side  inequality with $p/2$ implies that  
$$ \sum_{i=1}^k |\langle f,h_i\rangle|^p
=   \sum_{i=1}^k  ( |\langle f,h_i\rangle|^2)^{p/2}  
\ge k^{(1-\frac{p}{2})}  \left( \sum_{i=1}^k   |\langle f,h_i\rangle|^2\right)^{p/2}.  $$
By Proposition  \ref{min2} and  the left-side  inequality with $p/2$, we have that 
$$ \sum_{i=1}^k |\langle f,h_i\rangle|^p
\ge k^{(1-\frac{p}{2})}  \left( \sum_{i=1}^k   |\langle f,g_i\rangle|^2\right)^{p/2}
\ge  k^{(1-\frac{p}{2})}   \sum_{i=1}^k   |\langle f,g_i\rangle|^p,  $$
which completes the proof.
\end{proof}

\section*{Acknowledgment}

This research work was done when K. Berry, E. Evert, and S. Nghiem participated in the Central Michigan University NSF-REU program in the summer of 2012.

\bibliographystyle{amsplain}
\bibliography{References}

\end{document}